\documentclass[12pt,a4paper]{article}
\usepackage{amsfonts,amssymb,amsmath,graphicx,times}
\usepackage[english]{babel}
\usepackage{bm}

  \newtheorem{defi}{Definition}[section]
  \newtheorem{lem}[defi]{Lemma}
  
  \newtheorem{thm}[defi]{Theorem}
  \newtheorem{cor}[defi]{Corollary}

\newenvironment{nrtxt}{\begin{trivlist}\refstepcounter{defi}\item\textbf{\arabic{section}.\arabic{defi}~}}%!%
{\end{trivlist}}

\newcounter{abbildung} %% Z{\"a}hler fuer Abbildungen

\newcommand{\Matrixfeld}[4]{\left#1\!\begin{array}{*{#3}{c}}#4\end{array}\!\right#2}

\newcommand{\Mat}{\Matrixfeld()}

\newenvironment{proof}%[1]%
    {\begin{trivlist} \item {\emph{Proof}.}} %#1\/:}}%
    {\hfill$\square$\end{trivlist}} %%Hans {}\hfill erg{\"a}nzt und \/ gel{\"o}scht

\DeclareMathOperator{\Char}{{\mathrm{Char}}}%% eventuell klein schreiben,
\newcommand{\GF}{{\mathrm{GF}}}
\newcommand{\GL}{{\mathrm{GL}}}

\newcommand{\T}{{\mathrm{T}}}

\newcommand{\cL}{{\mathcal L}}
\newcommand{\cN}{{\mathcal N}}
\newcommand{\cO}{{\mathcal O}}

\newcommand{\cR}{{\mathcal R}}

\newcommand{\cV}{{\mathcal V}}

\newcommand{\PP}{{\mathbb P}}

\newcommand{\RR}{{\mathbb R}}

\newcommand{\bp}{{\bm p}}
\newcommand{\bq}{{\bm q}}
\newcommand{\bv}{{\bm v}}
\newcommand{\bw}{{\bm w}}

\newcommand{\by}{{\bm y}}

\newcommand{\bX}{{\bm X}}
\newcommand{\bY}{{\bm Y}}

\begin{document}
\sloppy

\title{The Betten-Walker spread\\ and Cayley's ruled cubic surface}
\author{Hans Havlicek \and Rolf Riesinger}
\date{\normalsize\emph{Unserem lieben Freund Dieter Betten zum
65.\ Geburtstag gewidmet}}

\maketitle

\begin{abstract}
We establish that, over certain ground fields, the set of osculating tangents
of Cayley's ruled cubic surface gives rise to a (maximal partial) spread which
is also a dual (maximal partial) spread. It is precisely the Betten-Walker
spreads that allow for this construction. Every infinite Betten-Walker spread
is not an algebraic set of lines, but it turns into such a set by adding just
one pencil of lines.

\par\noindent
2000 Mathematics Subject Classification. 51A40, 51M30, 14J26.

\par\noindent
Key words: Cayley's ruled cubic surface, osculating tangents, maximal partial
spread, maximal partial dual spread, algebraic set of lines.
\end{abstract}

\section{Introduction}

\begin{nrtxt}
In this paper we deal with a spread which was discovered independently, and
approximately at the same time, by D.~Betten \cite{betten-73a} and M.~Walker
\cite{walk-76a}.

\par

Betten used the concept of a \emph{transversal homeomorphism\/} in order to
describe and classify topological translation planes in terms of partitions of
the vector space $\RR^4$ into $2$-dimensional subspaces. What we call the
\emph{Betten-Walker spread\/} (BW-spread) is described in
\cite[Satz~3]{betten-73a}. Betten's paper contains also a short remark that the
construction of this spread works also for finite fields of characteristic
$\neq 3$ without a primitive third root of unity
\cite[pp.~338--339]{betten-73a}. We refer to \cite{knarr-95} or
\cite{salz+al-96} for the connection between spreads and translation planes; it
is due to J.~Andr{\'e} and was found independently by R.~H.~Bruck and R.~C.~Bose.

 \par
Walker adopted the projective point of view, which leads to spreads of lines in
a projective $3$-space. He focussed on the case of a finite ground field
$\GF(q)$, $q\equiv -1 \pmod 6$, and on the \emph{reguli\/} contained in the
spread. Thereby he laid the cornerstone for a concept which is now called the
\emph{Thas-Walker construction}. It links spreads with \emph{flocks} of
quadrics via the Klein mapping.

The BW-spread corresponds to a flock of a quadratic cone. In the finite case
this flock is due to C.~Fisher and J.~A.~Thas, who weakened Walker's condition
$q\equiv -1 \pmod 6$. We refer to \cite[pp.~334--338]{thas-95a} for further
details. Some authors use the term ``FTW-spread'' for a finite BW-spread.
 \par
The BW-spread and its corresponding translation plane were revisited by
A.~G.~Spera \cite{spera-86a}. The comprehensive paper by V.~Jha and
N.~L.~Johnson \cite{jha+j96a} (which should be read together with its second
part \cite{jha+j02}) contains more information about the BW-spread and its
associated flock. In both papers the existence of the BW-spread is established
for an arbitrary ground field $K$ with characteristic $\neq 3$ subject to the
condition that each element of $K$ has precisely one third root in $K$. We add
in passing that the BW-spread is among the ``likeable structures'' of
W.~M.~Kantor; see \cite{gevaert+j-88a}.
 \par
Finally, there is a neat connection, found by J.~A.~Thas, between flocks of
quadratic cones over finite fields and certain \emph{generalized quadrangles};
see \cite[p.~334]{thas-95a}, \cite{thas-01a}, and the references given there.
The infinite case was treated by F.~De~Clerck and H.~Van~Maldeghem
\cite{decl+vm-94a}. However, this connection with generalized quadrangles is
beyond the scope of the present paper. Let us just add the following remark: In
the finite case, the BW-spread corresponds to a generalized quadrangle
discovered by W.~M.~Kantor; cf.\ \cite[p.~398]{thas-95b}. Thus some authors
speak of the ``FTWKB generalized quadrangle'' in order to bring together the
names of all the involved mathematicians; see, for example,
\cite[p.~222]{okeefe+p-01a}.
\end{nrtxt}

\begin{nrtxt}
One aim of the present note is to present a short, direct, and self-contained
approach to the BW-spread, thereby establishing a connection with an algebraic
surface which was discovered already in the 19th century, namely \emph{Cayley's
ruled cubic surface}. According to \cite[p.~181]{mueller+k3-31} this name is
not completely appropriate, since M.~Chasles published his discovery of that
surface in 1861, three years before A.~Cayley.
\par
Our starting point is a Cayley surface $F$, say, in the projective $3$-space
over an arbitrary field $K$. At each simple point of $F$ there is a unique
osculating tangent other than a generator. The set of all those osculating
tangents, together with one particular line on $F$, gives then a set of lines,
say $\cO$, which easily turns out to be a spread if $K$ satisfies the
conditions mentioned above (characteristic $\neq3$, each element of $K$ has
precisely one third root in $K$). Moreover, when ``precisely'' is replaced with
``at most'' then $\cO$ is a maximal partial spread; see
Theorem~\ref{thm:S-eigenschaften} and cf.\ \cite[Theorem~6.3]{jha+j96a}. By our
approach the maximality of such a partial spread follows from the observation
that all points of a distinguished plane are incident with a line of the
partial spread.

By a classical result, there exists a duality which maps the Cayley surface (as
a set of points) onto the set of its tangent planes. Any mapping of this kind
fixes $\cO$, as a set of lines. Therefore, all our results hold together with
their dual counterparts. So, depending on the ground field, $\cO$ will be a
(maximal partial) spread and at the same time a (maximal partial) dual spread.

In case of characteristic three there is a \emph{line of nuclei} for the Cayley
surface. The existence of that line was noted by M.~de~Finis and
M.~J.~de~Resmini \cite{definis+d-83} without giving a geometric interpretation.
We show in Theorem~\ref{thm:par.netz} that the line of nuclei is the axis of a
parabolic congruence which contains all lines of $\cO$.
\end{nrtxt}

\begin{nrtxt}
The transversal homeomorphism used in \cite{betten-73a} to describe the
BW-spread is given in terms of polynomial functions. This raises the question
whether or not the BW-spread is \emph{algebraic}, i.e., its image under the
Klein mapping is an algebraic variety. In the finite case every set of points
is an algebraic variety (by \cite[Lemma~3.5~(a)]{hirs+s-94a}, even a
hypersurface), whence we exclude that case from our investigation. On the
other hand, infinite algebraic spreads seem to be rare. The only examples known
to the authors are the regular spreads (or, in other words, the elliptic linear
congruences) and some spreads found by the second author; see
\cite[Table~1]{ries-05a}. Unfortunately, our hope to find another example of an
algebraic spread did not come true. However, the BW-spread is very close to
being algebraic. We establish in Theorem~\ref{thm:algebraisch} that the union
of the BW-spread and one pencil of lines is algebraic. More precisely, the
Klein image of that set is the smallest algebraic variety containing the Klein
image of the BW-spread (Theorem~\ref{thm:abgeschlossen}). When looking for
equations describing that variety (in terms of Pl{\"u}cker coordinates) the thesis
of R.~Koch \cite{koch-68} turned out extremely useful, even though we could not
directly implement his results in our work. It is worth mentioning that the
BW-spread (over $\RR$) is ubiquitous in Koch's thesis under the German name
``Schmiegtangentenkongruenz'' (congruence of osculating tangents), but the
property of being a spread never seems to be mentioned in the text. Likewise,
the authors were unable to find a remark on this property in the older
literature on the Cayley surface.
\end{nrtxt}

\section{The Cayley surface}\label{se:cayley}

\begin{nrtxt}
We consider the three-dimensional projective space $\PP_3(K)$ over a
commutative field $K$. As we shall use column vectors, a point has the form $K
\bp$ with $(0,0,0,0)^{\T} \ne \bp = (p_0,p_1,p_2,p_3)^{\T} \in K^{4 \times 1}$.
The set of lines of $\PP_3(K)$ is written as $\cL$.

\par

Let $\bX:=(X_0,X_1,X_2,X_3)$ be a family of independent indeterminates over
$K$. We refer to \cite[pp.~48--51]{hirs-98} for those basic notions of
algebraic geometry which will be used in this paper. However, in contrast to
\cite{hirs-98}, we write
\begin{equation*}
 \cV\big(g_1(\bX),g_2(\bX),\ldots,g_r(\bX)\big)
 :=\big\{K\bp\in\PP_3(K)\mid g_1(\bp)=g_2(\bp)=\cdots=g_r(\bp)=0\big\}
\end{equation*}
for the set of $K$-rational points of the variety given by homogeneous
polynomials $g_1(\bX),\; g_2(\bX),\; \ldots,\; g_r(\bX) \in K[\bX]$.

\par

Each matrix $M=(m_{ij})_{0\leq i,j\leq 3}\in\GL_4(K)$ acts on the column space
$K^{4 \times 1}$ by multiplication from the left hand side and therefore as a
projective collineation on $\PP_3(K)$. Moreover, $M$ acts as a $K$-algebra
isomorphism on $K[\bX]$ via $X_i \mapsto \sum_{j=0}^{3} m_{ij}X_j$ for
$i\in\{0,1,2,3\}$. Given a form $g(\bX)\in K[\bX]$ and its image under $M$, say
$h(\bX)$, the collineation induced by $M$ takes $\cV\big(h(\bX)\big)$ to
$\cV\big(g(\bX)\big)$.

\par

In what follows the plane $\omega:=\cV(X_0)$ will be considered as \emph{plane
at infinity}; thus we turn $\PP_3(K)$ into a projectively closed affine space.
\end{nrtxt}

\begin{nrtxt}
We refer to \cite{bertini-23}, \cite{brau-64}, \cite{brau-67c},
\cite{gmai+h-04z}, \cite{rosati-56}, and \cite{sauer-37} for the definition and
basic properties of \emph{Cayley's ruled cubic surface\/} or, for short, the
\emph{Cayley surface}. It is, to within projective collineations, the point set
$F:=\cV\big(f(\bX)\big)$, where
\begin{equation*}
 f(\bX):=X_0X_1X_2-X_1^3-X_0^2X_3 \in
 K[\bX].
\end{equation*}
Let $\partial_i:=\frac{\partial}{\partial X_i}$. Hence we obtain
\begin{equation}\label{eq:partial}
\begin{array}{rcl@{\quad}rcl}
    \partial_0 f(\bX) &=& X_1X_2-2X_0X_3,&
    \partial_1 f(\bX) &=& X_0X_2-3X_1^2,\\
    \partial_2 f(\bX) &=& X_0X_1,&
    \partial_3 f(\bX) &=& -X_0^2.
\end{array}
\end{equation}
These partial derivatives vanish simultaneously at $(p_0,p_1,p_2,p_3)^{\T} \in
K^{4 \times 1}$ if, and only if, at least one of the following conditions
holds:
\begin{eqnarray}
    \label{eq:sing-punkte}
    & p_0=p_1=0; &\\
    \label{eq:knotengerade}
    & p_0=p_2=0 \mbox{ and } \Char K=3. &
\end{eqnarray}
The parametrization
\begin{equation*}
 K^2\to \PP_3(K): (u_1,u_2) \mapsto K(1,u_1,u_2,u_1u_2-u_1^3)^{\T}=:P(u_1,u_2)
\end{equation*}
is injective, and its image coincides with $F\setminus\omega$ (the affine part
of $F$). According to (\ref{eq:partial}), all points of $F\setminus \omega$ are
\emph{simple}. The \emph{tangent plane\/} at $P(u_1,u_2)$ equals
\begin{equation}\label{eq:tangebene}
   \cV\left( (2u_1^3-u_1u_2)X_0 + (-3u_1^2+u_2)X_1 + u_1X_2 - X_3\right).
\end{equation}
The points subject to (\ref{eq:sing-punkte}) comprise the line
$\cV(X_0,X_1)=F\cap\omega=:g_{\infty}$. They are easily seen to be \emph{double
points\/} of $F$. The \emph{tangent cone\/} (or \emph{tangent space\/}
\cite[p.~49]{hirs-98}) at a point $U:=K(0,0,s_2,s_3)^\T$, $(s_2,s_3)\neq
(0,0)$, is
\begin{equation}\label{eq:tangential-unendlich}
  \cV\big(X_0(s_2X_1-s_3X_0)\big);
\end{equation}
this is either a pair of distinct planes (if $U\neq Z:=K(0,0,0,1)^\T$) or a
repeated plane (if $U=Z$). We call each of these planes a \emph {tangent
plane\/} at $U$. The point $Z$ is a so-called \emph{pinch point\/}
\cite[p.~76]{mueller+k3-31}, and its tangent plane is $\omega$. See
Figure~\ref{fig:cf-schlaufen} which displays the Cayley surface in $\PP_3(\RR)$
in an affine neighbourhood of $Z$. (The plane $\cV(X_3)$ is at infinity in this
illustration.)

\par
The tangent plane of $F$ at $P(0,0)$ is $\cV(X_3)$; this plane meets $F$ along
the line $\cV(X_1,X_3)$ and the parabola
\begin{equation}\label{eq:leitkegelschnitt}
    l:=\cV(X_0X_2-X_1^2,X_3).
\end{equation}
For each $(s_0,s_1)\in K^2\setminus\{(0,0)\}$ the line
\begin{equation*}
 g(s_0,s_1):=K(s_0^2,s_0s_1,s_1^2,0)^{\T} + K(0,0,s_0,s_1)^{\T}
\end{equation*}
is a \emph{generator\/} of $F$. There are no other lines on $F$. The line
$g(0,1)=g_{\infty}$ is not only a generator of $F$, but also a
\emph{directrix}, as it has non-empty intersection with every generator. Each
point of $g_{\infty}$, except the point $Z$, is on precisely two generators of
$F$; each affine point of $F$ is incident with precisely one generator
(Figure~\ref{fig:cf-erz}).
%%%%%%%%%%%%%%%%%%%%%%
\begin{center}
%%%%%%%%%%%%%%%%%%%%%%
\begin{minipage}[t]{6.0cm}
\unitlength1.0cm\small
\begin{picture}(6,2.66)
 \put(0.0,0.0){\includegraphics[width=6\unitlength]{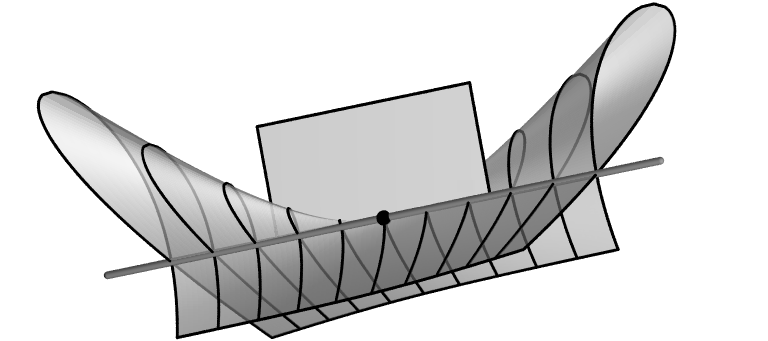}}
%% \put(0.0,0.0){\includegraphics[width=7\unitlength]{./gitter.eps}}
 \put(2.85,1.1){$Z$}
 \put(5.3,1.3){$g_\infty$}
 \put(2.15,1.45){$\omega$}
 \put(0.,2.0){$F$}
\end{picture}
\refstepcounter{abbildung}\label{fig:cf-schlaufen}
  \centerline{Figure \ref{fig:cf-schlaufen}}
\end{minipage}
%%%%%%%%%%%%%%%%%%%%%%
\begin{minipage}[t]{6.0cm}
\unitlength1.0cm\small
\begin{picture}(6,2.66)
 \put(0.0,0.0){\includegraphics[width=6\unitlength]{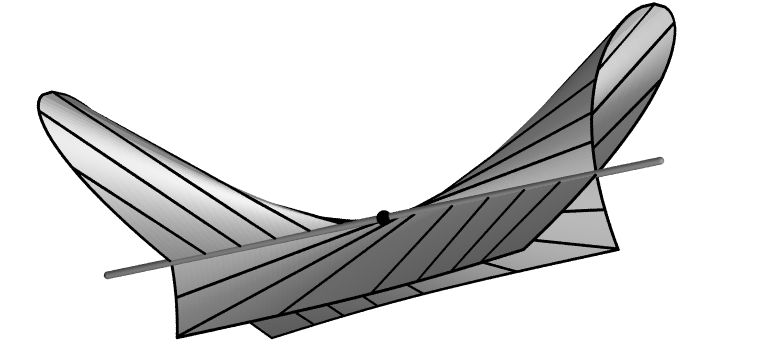}}
 \put(2.85,1.1){$Z$}
 \put(5.3,1.3){$g_\infty$}
 \put(0.,2.0){$F$}
\end{picture}
\refstepcounter{abbildung}\label{fig:cf-erz}
  \centerline{Figure \ref{fig:cf-erz}}
\end{minipage}
%%%%%%%%%%%%%%%%%%%%%%
\end{center}
%%%%%%%%%%%%%%%%%%%%%%

Next we describe the automorphic projective collineations of $F$: The set of
all matrices
\begin{equation*}
  M_{a,b,c}:= \Mat4{
  1&0&0&0\\
  a&c&0&0\\
  b&3\,ac&{c}^{2}&0\\
  ab-{a}^{3}&bc&a{c}^{2}&{c}^{3}
  }
\end{equation*}
where $a,b \in K$ and $c \in K\setminus\{0\}$ is a group, say $G$, under
multiplication. Each matrix in $G$ leaves invariant the cubic form $f(\bX)=
X_0X_1X_2-X_1^3-X^{2}_0X_3$ to within the factor $c^3$. Consequently, the group
$G$ acts on $F$ as a group of projective collineations. Under the action of
$G$, the points of $F$ fall into three orbits: $F\setminus\omega$,
$g_\infty\setminus\{Z\}$, and $\{Z \}$. Except for the case when $|K|\leq 3$,
the group $G$ yields all projective collineations of $F$; see
\cite[Section~3]{gmai+h-04z}.

Observe that the following holds irrespective of the characteristic of $K$.
\end{nrtxt}

\begin{lem}\label{lem:dualitaet}
There exists a duality which maps the set of points of the Cayley surface $F$
onto the set of its tangent planes. Thus the set of all tangent planes of $F$
is a Cayley surface in the dual projective space.
\end{lem}
\begin{proof}
By (\ref{eq:tangebene}) and (\ref{eq:tangential-unendlich}), a plane
$\cV\left(\sum_{i=0}^{3}a_iX_i\right)$, where $a_i\in K$, is a tangent plane of
$F$ if, and only if, $a_1a_2a_3 - a_2^3 - a_0a_3^2 =0$. Consequently, the
linear bijection
\begin{equation*}
  K^{4\times 1}\to K^{1\times 4} : (x_0,x_1,x_2,x_3)^\T \mapsto
                                   (x_3,x_2,x_1,x_0)
\end{equation*}
gives a duality of $\PP_3(K)$ with the required properties.
\end{proof}
We note that the duality from the above takes, for all $(u_1,u_2)\in K^2$, the
point $P(u_1,u_2)$ to the tangent plane at the point $P(-u_1,3u_1^2-u_2)$.

\section{Osculating tangents and the Betten-Walker spread}\label{se:osk}

\begin{nrtxt}
If a line $t$ meets $F$ at a simple point $P$ with multiplicity $\geq 3$ then
it is called an \emph{osculating tangent\/} at $P$. Such a tangent line is
either a generator or it meets $F$ at $P$ only. In the latter case it will be
called a \emph{proper osculating tangent\/} of $F$. Observe that we are not
dealing with those lines which meet $F$ with multiplicity $\geq 3$ at a double
point. In fact, $g_\infty$ is the set of double points, and at any point $U\in
g_\infty$, the lines meeting $F$ at $U$ with multiplicity $\geq 3$ comprise two
pencils (one pencil if $U=Z$) lying in the two tangent planes at $U$ (only
tangent plane at $U=Z$); see formula (\ref{eq:tangential-unendlich}). The
following is part of the folklore:
\end{nrtxt}

\begin{lem}\label{lem:schmiegtangfernpunkt}
At each point $P(u_1,u_2) \in F \setminus g_{\infty}$ there is a unique proper
osculating tangent, namely the line which joins $P(u_1,u_2)$ with the point
$K(0,1,3u_1,u_2)^\T$.
\end{lem}
\begin{proof}
Let $(u_1,u_2)=(0,0)$. The tangent plane at $P(0,0)$ is $\cV(X_3)$. Any proper
osculating tangent through $P(0,0)$ is necessarily incident with this plane,
and it meets $F$ at $P(0,0)$ only. By (\ref{eq:leitkegelschnitt}), only the
tangent $t$ of the parabola $l$ at $P(0,0)$ can be a proper osculating tangent,
since every other tangent of $F$ at $P(0,0)$ meets $l$ residually at a point
$\neq P(0,0)$. The point at infinity of $t$ is $K(0,1,0,0)$. It is
straightforward to verify that $t$ meets $F$ at $P(0,0)$ with multiplicity
three. By the action of the matrix $M_{u_1,u_2,1}\in G$ the assertion follows
for any point $P(u_1,u_2) \in F \setminus g_{\infty}$.
\end{proof}

\begin{thm}\label{thm:S-eigenschaften}
Let
\begin{equation*}
    \cO:=\{t\in \cL\mid t\mbox{ is a proper osculating tangent of }F \}\cup\{g_\infty\}.
\end{equation*}
This set of lines has the following properties.
\begin{enumerate}
\item $\cO$ is a partial spread of\/ $\PP_3(K)$ if, and only if, $\Char K \neq
3$ and $K$ does not contain a third root of unity other than $1$.

\item If $\cO$ is a partial spread then it is maximal, i.e., it is not a proper
subset of any partial spread of\/ $\PP_3(K)$.

\item $\cO$ is a covering of\/ $\PP_3(K)$ if, and only if, $\Char K\neq 3$ and
each element of $K$ has a third root in $K$.
\end{enumerate}
\end{thm}
\begin{proof}
(a) It is immediate from Lemma \ref{lem:schmiegtangfernpunkt} that all proper
osculating tangents of $F$ are skew to $g_\infty$. So it suffices to discuss
whether or not two distinct proper osculating tangents of $F$ have a point in
common. As the group $G$ acts transitively on $F\setminus g_\infty$, all proper
osculating tangents of $F$ are in one orbit of $G$. So it is enough consider
the osculating tangents at distinct points $P(0,0)$ and $P(u_1,u_2)$. By Lemma
\ref{lem:schmiegtangfernpunkt}, these lines are skew if, and only if,
\begin{equation}\label{eq:schnittbedingung}
    \det \Mat4{
  1 & 0 & 1 & 0 \\
  0 & 1 & u_1 & 1 \\
  0 & 0 & u_2 & 3u_1  \\
  0 & 0 & u_1u_2-u_1^3 & u_2  } = u_2^2-3u_1^2u_2+3u_1^4\neq 0.
\end{equation}
If $u_1=0$ then $u_2\neq 0$, whence (\ref{eq:schnittbedingung}) holds
irrespective of the ground field. Otherwise we substitute $u_2=(2+y)u_1^2$ with
$y\in K$. Hence (\ref{eq:schnittbedingung}) turns into $u_1^4(y^2+y+1)\neq 0$.
Observing $(X^2+X+1)(X-1)=X^3-1\in K[X]$, we see that $X^2+X+1$ has a zero in
$K$ precisely when the following holds: Either $\Char K=3$, since in this case
$X^2+X+1=(X-1)^2$, or $\Char K\neq 3$ and there exists a third root of unity
$w\neq 1$ in $K$, since then $1^2+1+1\neq 0$.

(b) We infer from (a) that $\Char K\neq 3$. Thus Lemma
\ref{lem:schmiegtangfernpunkt} implies that each point at infinity is incident
with a proper osculating tangent or just with the line $g_{\infty}$. As every
line of $\PP_3(K)$ has a point in common with the plane at infinity and each
point of the plane at infinity is on a line of the partial spread, the partial
spread $\cO$ is maximal.

(c) First, let $\Char K=3$. It suffices to show that $\cO$ cannot be a
covering of $\PP_3(K)$. By Lemma~\ref{lem:schmiegtangfernpunkt}, all proper
osculating tangents meet the line
\begin{equation}\label{def:knoten}
  n:=\cV(X_0,X_2).
\end{equation}
Clearly, there exists a point in $\omega\setminus(n\cup g_\infty)$. This point
is not incident with any line of $\cO$.

\par

Next, assume $\Char K\neq 3$. By the proof of (b), we may restrict ourselves to
affine points. A point $K(1,p_1,p_2,p_3)$ is on a line of $\cO$ if, and only
if, there is a pair $(u_1,u_2) \in K^2$ and an $s\in K$ such that
\begin{equation*}
    (1,p_1,p_2,p_3)^{\T}=(1,u_1,u_2,u_1u_2-u_1^3)^{\T}+s(0,1,3u_1,u_2)^{\T}.
\end{equation*}
So we obtain the following system of equations in the unknowns $u_1,u_2,s$:
\begin{equation*}
        u_1  =  p_1-s,\;\;
        u_2  =  p_2-3s(p_1-s),\;\;
        s^3  =  p_3-(p_1p_2-p_1^3).
\end{equation*}
This system has a solution precisely when $p_3-(p_1p_2-p_1^3)$ has a third root
in $K$. As $p_3-(p_1p_2-p_1^3)$ can assume any value in $K$, the assertion
follows.
\end{proof}
%%%%%%%%%%%%%%%%%%%

\begin{nrtxt}
By the above, an affine point lies on a line of $\cO$ if, and only if, it can
be written in the form $K(1,p_1,p_2,p_1p_2-p_1^3 + s^3)$ with $p_1,p_2,s\in K$.

The results of Theorem~\ref{thm:S-eigenschaften} were established in
\cite[Theorem~6.1]{jha+j96a} and \cite[Teorema~1]{spera-86a} in a completely
different way. In those papers the reader will also find conditions for a field
$K$ to meet one of the algebraic conditions given in (a), (b), or (c).

We noted in Lemma \ref{lem:dualitaet} that $F$ admits a duality which is easily
seen to fix $\cO$, as a set of lines. Hence the dual counterparts of the
characterizations given in Theorem~\ref{thm:S-eigenschaften} hold as well. Thus
$\cO$ is a (maximal partial) spread if, and only if, it is a (maximal partial)
dual spread. Observe that the point $Z=K(0,0,0,1)^\T$ takes over the role of
the plane $\omega$ in the dual setting.
\end{nrtxt}

\begin{nrtxt}
Let $\Char K=3$. Recall from (\ref{def:knoten}) that $n=\cV(X_0,X_2)$. By
(\ref{eq:knotengerade}), every point of $n\setminus\{ Z \}$ is a \emph{nucleus}
of $F$, i.e. a point off $F$, where all partial derivatives (\ref{eq:partial})
vanish; see \cite[p.~50]{hirs-98}. We refer also to
\cite[Proposition~3.17]{definis+d-83}, where nuclei are defined in a slightly
different way (including double points of $F$). Even though $Z$ is not a
nucleus according to our definition, we shall refer to $n$ as being the
\emph{line of nuclei}. We established in the proof of
Theorem~\ref{thm:S-eigenschaften}~(c) that all proper osculating tangents meet
the line of nuclei. This result will be improved in Theorem~\ref{thm:par.netz}.
\end{nrtxt}

\begin{nrtxt}
Let $K=\RR$ so that $\cO$ is a spread. In order to show that this is in fact
the BW-spread, as described in \cite[Satz~3]{betten-73a}, we apply the
collineation
\begin{equation*}
  \alpha: \PP_3(K)\to\PP_3(K) : K(x_0,x_1,x_2,x_3)^\T
             \mapsto
             K\left(x_0,x_1,\frac{x_2}{3}, \frac{x_3}{3}\right)^\T
\end{equation*}
which fixes the line $g_\infty$. By Lemma \ref{lem:schmiegtangfernpunkt}, any
line of $\alpha(\cO\setminus\{g_\infty\})$ has the form
\begin{equation*}
    K\left(1,u_1,\frac{u_2}{3},\frac{u_1u_2}{3} - \frac{u_1^3}{3}\right)^\T
    +
    K\left(0,1,u_1,\frac{u_2}{3}\right)^\T
\end{equation*}
with $(u_1,u_2)\in \RR^2$. By joining this line with $Z=K(0,0,0,1)^\T$ and
$K(0,0,1,0)^\T$, we obtain two distinct planes with equations
\begin{equation*}
  x_2 = \left(\frac{u_2}3 - u_1^2\right)x_0 + u_1x_1
  \mbox{ and }
  x_3 = -\frac{u_1^3}{3}\,x_0 + \frac{u_2}{3}\,x_1,
\end{equation*}
respectively. The substitutions $x_0=:x$, $x_1=:y$, $x_2=:u$, $x_3=:v$,
${u_2}/{3}-u_1^2=:t$, and $u_1=:s$ turn these equations into
\begin{equation*}
  u = tx + sy
  \mbox{ and }
  v = -\frac{s^3}{3}x + (s^2+t)y.
\end{equation*}
These are the formulas from \cite[Satz~3]{betten-73a}. In particular, we have
the transversal homeomorphism of $\RR^2$ with $(t,s) \mapsto (-{s^3}/{3},
s^2+t)$.

%%%%%%%%%%%%%%%%%%%%%% Beschriftung schaut seitlich dr{\"u}ber
\begin{center}
\unitlength1.0cm\small
\begin{picture}(6.0,7.68)
 \put(0.0,0.0){\includegraphics[width=6\unitlength]{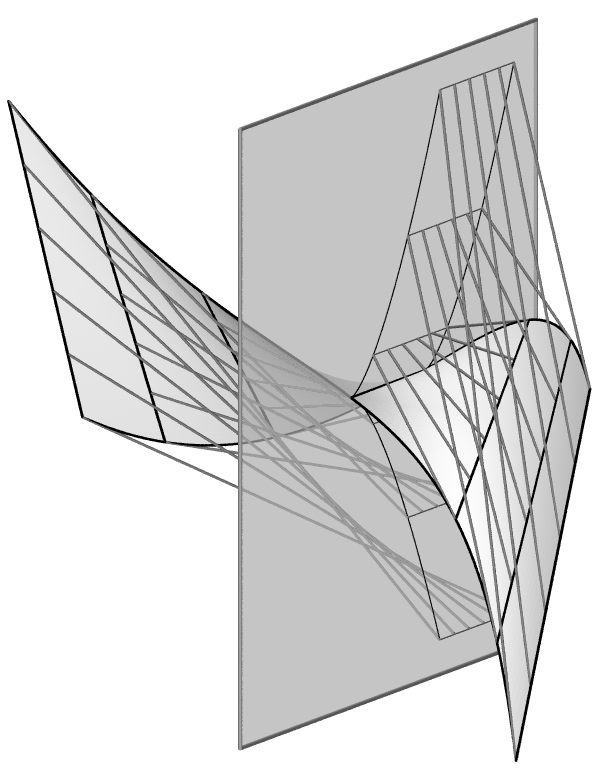}}
% \put(0.0,0.0){\includegraphics[width=7\unitlength]{./gitter.eps}}
 \put(2.5,6.0){$\cV(X_1)$}
 \put(0.0,5.0){$F$}
 \put(5.5,1.0){$g(1,s)$}
 \put(5.6,5.6){$\cR^-(s)$}
\end{picture}
\par
\refstepcounter{abbildung}\label{fig:CFHP} Figure \ref{fig:CFHP}
\end{center}
%%%%%%%%%%%%%%%%%%%%%%

It is also easy to see that our results coincide with \cite{walk-76a}, where
homogeneous coordinates with indices running from $1$ to $4$, say
$x'_1,x'_2,x'_3,x'_4$, were used. The appropriate transformation from our
coordinates $x_0,x_1,x_2,x_3$ is given by $x'_1=x_3 /3$, $x'_2=x_2/3$,
$x'_3=x_1$, and $x'_4=x_0$. The group $S_1$ used in \cite{walk-76a} is a
subgroup of our group $G$, whereas the reguli $\mathfrak{R}_i$ from
\cite{walk-76a} arise in our setting as follows: Take the set of all proper
osculating tangents at the points of a generator $g(1,s)$, $s\in K$, together
with $g_\infty$. This is easily seen to be a regulus, say $\cR^-(s)$, which
clearly is contained in $\cO$. In affine terms each such regulus is one family
of generators on a hyperbolic paraboloid. These hyperbolic paraboloids have
$g_\infty$ as a common generator and they share a common tangent plane at each
point of $g_\infty$. Thus, for example, each such paraboloid meets the plane
$\cV(X_1)\supset g_\infty$ residually in a line; all these lines are parallel,
as they pass through $K(0,0,1,0)^\T$. This is depicted in
Figure~\ref{fig:CFHP}.
\end{nrtxt}

\begin{nrtxt}\label{nr:reguli}
Assume $\Char K\neq 3$. Then the lines of $\cO$ other than $g_\infty$ define
(by intersection) an injective mapping $\omega\setminus g_\infty\to
\cV(X_1)\setminus g_\infty$; compare with the construction of a spread via a
transversal mapping due to N.~Knarr \cite[pp.~26--29]{knarr-95}. An
illustration is given in Figure~\ref{fig:transversal}, where temporarily
$\cV(X_0+X_1)$ takes over the role of the plane at infinity. (The curves in
$\cV(X_1)$ are semicubical parabolas.)
\begin{center}
\unitlength1.0cm\small
\begin{picture}(8,6)
 \put(0.0,0.0){\includegraphics[width=8\unitlength]{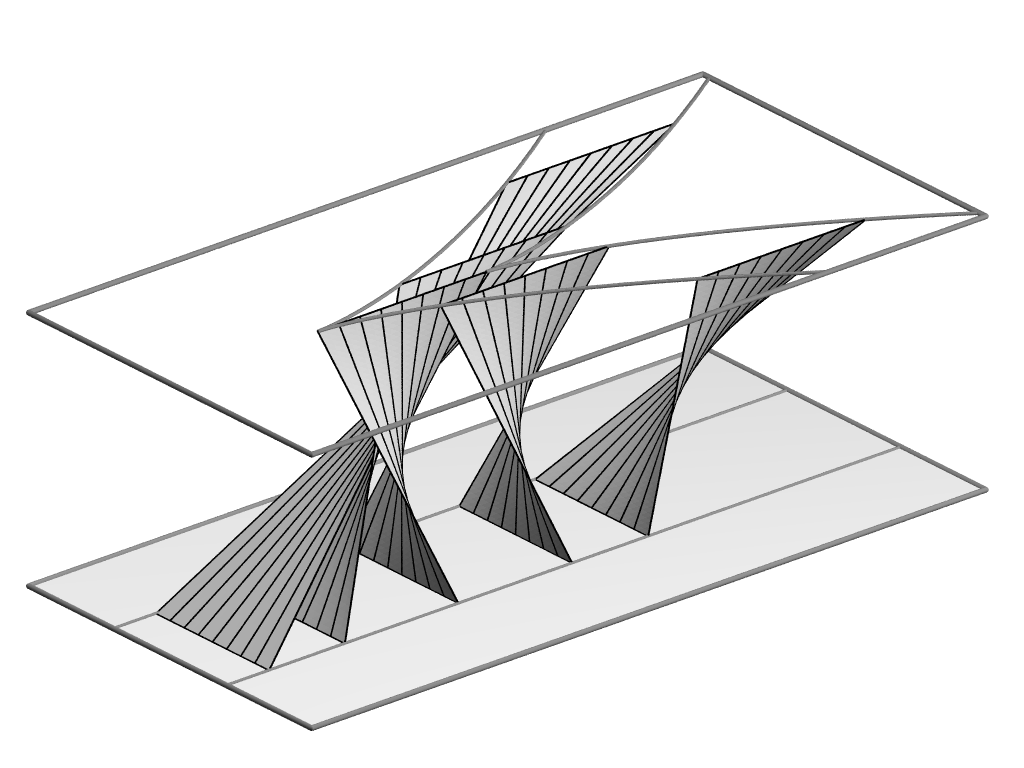}}
%% \put(0.0,0.0){\includegraphics[width=7\unitlength]{./gitter.eps}}
 \put(0.8,1.3){$\omega$}
 \put(0.8,3.4){$\cV(X_1)$}
\end{picture}
\par
\refstepcounter{abbildung}\label{fig:transversal} Figure \ref{fig:transversal}
\end{center}
\end{nrtxt}

\section{The Klein image of the Betten-Walker spread}\label{se:klein}

\begin{nrtxt}
In terms of coordinates, the exterior square $K^{4 \times 1} \wedge K^{4 \times
1}$ coincides with $K^{6 \times 1}$ by setting
\begin{equation*}
  \bp \wedge \bq= (p_0,p_1,p_2,p_3)^{\T} \wedge
(q_0,q_1,q_2,q_3)^{\T}=(y_{01},y_{02},y_{03},y_{12},y_{13},y_{23})^{\T}
\end{equation*}
where $y_{ij}=p_iq_j-p_jq_i$. Given that $\bp, \bq$ are linearly independent
the entries of the column vector $(y_{01},y_{02},\dots,y_{23})^{\T}$ are the
well known Pl{\"u}cker coordinates of the line $K\bp+K\bq$. The Klein mapping
$\kappa: \cL \rightarrow Q: K\bp+K\bq \mapsto K(\bp \wedge \bq)$ is a bijection
from the line set $\cL$ of $\PP_3(K)$ onto Klein quadric $Q:=\cV(k(\bY))\subset
\PP_5(K)$, where $\bY=( Y_{01},Y_{02},\ldots,Y_{23})$ denotes a family of six
independent indeterminates over $K$ and
\begin{equation*}
   k(\bY):= Y_{01}Y_{23}-Y_{02}Y_{13}+Y_{03}Y_{12}.
\end{equation*}
The polarity of the Klein quadric will be denoted by $\perp$. Observe that
$\perp$ is symplectic if, and only if, $\Char K=2$. Table 15.10 in
\cite[pp.~29--31]{hirs-85} contains all the information on the Klein mapping
which we shall use in this section without further reference.
\par
A set of lines in $\PP_3(K)$ is said to be \emph{algebraic\/} if its Klein
image is the set of $K$-rational points of an algebraic variety in $\PP_5(K)$.
\end{nrtxt}

\begin{nrtxt}
Let us first calculate the Klein image of the set of generators of $F$: We
obtain, for all $(s_0,s_1) \in K^2 \setminus \{(0,0)\}$, that
\begin{equation}\label{eq:erzeugendenbild}
    \kappa(g(s_0,s_1)) = K(0,s_0^3,s_0^2s_1,s_0^2s_1,s_0s_1^2,s_1^3)^{\T}
    \in \PP_5(K).
\end{equation}
So we get a twisted cubic \cite[Chapter~21]{hirs-85} lying in the
three-dimensional subspace
\begin{equation*}
    C:=\cV(X_{01},X_{03}-X_{12})\subset\PP_5(K).
\end{equation*}
The intersection $C\cap Q$ is a quadratic cone with vertex
$W_\infty:=\kappa(g_{\infty})=K\bw_\infty$, where
\begin{equation}\label{def:w.inf}
    \bw_\infty:=(0,0,0,0,0,1)^{\T}.
\end{equation}
This cone is the Klein image of a parabolic linear congruence with axis
$g_{\infty}$, which contains all generators of $F$. The subspace $C^\perp$ is
the line spanned by $W_\infty$ and $W:=K\bw$, where
\begin{equation*}
    \bw:=(0,0,1,-1,0,0)^\T.
\end{equation*}
This line meets the Klein quadric at $W_\infty$ only. We have $C^\perp\cap
C=\{W_\infty\}$ for $\Char K\neq 2$, but $C^\perp\subset C$ otherwise.

In the subsequent theorem we exhibit algebraic equations which are satisfied by
the Klein image of $\cO$; we shall explain in \ref{nr:gleichungen} how these
equations were found.
\end{nrtxt}

\begin{thm}\label{thm:algebraisch}
Suppose $\Char K\neq 3$. Let $\cO$ be given as in
Theorem~\emph{\ref{thm:S-eigenschaften}}, and let $\cL[Z,\omega]$ be the pencil
of lines in the plane $\omega=\cV(X_0)$ with centre $Z=K(0,0,0,1)^\T$. Consider
the polynomials
\begin{eqnarray}
    h_1(\bY)&:=&3 Y_{01}( Y_{12}+Y_{03}) -Y_{02}^{2},\label{eq:h1}\\
    h_2(\bY)&:=&3Y_{02}Y_{13}- ( Y_{12}+Y_{03}) ^{2},\label{eq:h2}\\
    h_3(\bY)&:=&9Y_{01}Y_{13}-Y_{02} ( Y_{12}+Y_{03}).\label{eq:h3}
\end{eqnarray}
Then $\kappa(\cO\cup\cL[Z,\omega])$ equals the intersection of the variety
\begin{equation*}
   J:= \cV(h_1(\bY),h_2(\bY),h_3(\bY))\subset\PP_5(K)
\end{equation*}
with the Klein quadric $Q=\cV(k(\bY))$.
\end{thm}
\begin{proof}
(a) For all $(u_1,u_2)\in K^2$, the Klein image of the only proper osculating
tangent at $P(u_1,u_2)$ is the point with coordinates
\begin{equation}\label{eq:klein-schmiegtang}
  (1,3u_1,u_2,3u_1^2-u_2,u_1^3,3u_1^4-3u_1^2u_2+u_2^2)^\T.
\end{equation}
The Klein image of the pencil $\cL[Z,\omega]$ is the line spanned by
$K(0,0,0,0,1,0)$ and $K\bw_\infty=\kappa(g_\infty)$; see (\ref{def:w.inf}). Now
a direct verification shows $\kappa(\cO\cup\cL[Z,\omega])\subset (J\cap Q)$.

(b) In order to show $(J\cap Q)\subset \kappa(\cO\cup\cL[Z,\omega])$ we
determine all vectors
\begin{equation*}
    \by:=(y_{01},y_{02},y_{03},y_{12},y_{13},y_{23})^\T \in K^{6\times 1}
\end{equation*}
subject to $h_1(\by)=h_2(\by)=h_3(\by)=k(\by)=0$.

In a first step we determine all such vectors with $y_{01}\neq 0$. Without loss
of generality we may assume $y_{01}=1$; also we let $y_{02}=:u_1$ and
$y_{03}=:u_2$. From $h_1(\by)= 3 y_{12}+3 u_2 - 9u_1^2$ follows $y_{12}=3 u_1^2
- u_2$ which can be substituted in $h_3(\by) = 9 y_{13}- 9 u_1^3$. This gives
$y_{13}=u_1^3$. We calculate $k(\by)=y_{23}-3u_1^{4}+u_2 ( u_1^{2}-u_2)$ which
yields $y_{23} = 3u_1^{4}-3u_1^{2} u_2+u_2^{2}$. Altogether, we obtain
precisely the vectors given in (\ref{eq:klein-schmiegtang}), whence
$h_2(\by)=0$ holds too.

The second step is to look for all solutions with $y_{01}=0$. We infer from
$h_1(\by)=-y_{02}^2$ that $y_{02}=0$, from which we obtain
$h_2(\by)=-(y_{12}+y_{03})^2$. So $y_{12}=-y_{03}$. Now $k(\by)=-y_{03}^2$
gives $y_{03}=-y_{12}=0$. Summing up we obtain
\begin{equation*}
    \by:=(0,0,0,0,y_{13},y_{23})^\T\in K^{6\times 1}
\end{equation*}
which is either the zero vector or a representative of a point in
$\kappa(\cL[Z,\omega])$. Consequently, also $h_3(\by)=0$ is satisfied.
\end{proof}

\begin{nrtxt}\label{nr:gleichungen}
Let us shortly describe how the polynomials $h_i(\bY)$ were found: We noted
already in \ref{nr:reguli} that all proper osculating tangents at the points of
a generator $g(1,s)$, $s\in K$, together with $g_\infty$ form a regulus
$\cR^-(s)\subset \cO$. The opposite regulus $\cR^+(s)$, say, contains the
generator $g(1,s)$. Both reguli lie on a hyperbolic paraboloid which is known
in differential line geometry under the name \emph{Lie quadric\/} of $F$ along
the generator $g(1,s)$; cf., among others, \cite[pp.~33--37]{hosch-71a} or
\cite[pp.~67--68]{sauer-37}. In Figure~\ref{fig:reguli} some reguli $\cR^-(s)$
are visualized in an affine neighbourhood of the point $Z$. In this picture the
plane $\cV(X_3)$ appears at infinity. See also Figures~\ref{fig:CFHP} and
\ref{fig:transversal} for a different view of these reguli.
%%%%%%%%%%%%%%%%%%%%%%
\begin{center}%% Auf 10cm strecken ist 768 breit, so wie die anderen!
\unitlength1.3cm\small
\begin{picture}(6,3.61)
 \put(0.0,0.0){\includegraphics[width=6\unitlength]{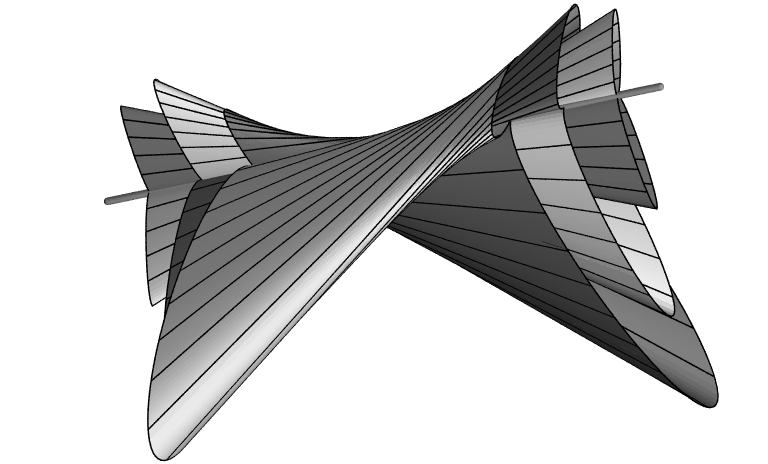}}
%% \put(0.0,0.0){\includegraphics[width=7\unitlength]{./gitter.eps}}
 \put(2.95,3.05){\makebox(0,0)[b]{$Z$}}%% zentriert nach dem Pfeil
 \put(2.95,2.95){\vector(0,-1){0.3}} %% Pfeil
 \put(5.3,3.0){$g_\infty$}
\end{picture}
\par
\refstepcounter{abbildung}\label{fig:reguli} Figure \ref{fig:reguli}
\end{center}
%%%%%%%%%%%%%%%%%%%%%%
Given a point $\kappa(g(1,s))$ of the twisted cubic (\ref{eq:erzeugendenbild})
let us denote by $\pi(s)$ the osculating plane at this point. The plane
$\pi(s)$ meets the cone $C\cap Q$ in a conic which is the Klein image of the
regulus $\cR^+(s)$. The Klein image of its opposite regulus $\cR^-(s)$ is a
conic lying in the plane $\pi(s)^\perp\supset C^\perp$. We choose the
three-dimensional subspace
\begin{equation*}
   B:=\cV(Y_{03},Y_{23})
\end{equation*}
which is skew to $C^\perp$. Thus $\pi(s)^\perp$ meets $B$ at a unique point. We
obtain this point by letting $u_1=s$ in (\ref{eq:klein-schmiegtang}) and by
projecting through $C^\perp$ to $B$. This gives (for every $u_2\in K$)
\begin{equation}\label{eq:leitkubik}
    K(1,3\,s,0,3\,{s}^{2},{s}^{3},0)^\T.
\end{equation}
Now we consider $s\in K$ to be variable. Up to the exceptional case when $\Char
K=3$, the planes $\pi(s)$ belong to a cubic developable, so that the points of
intersection $\pi(s)^\perp\cap B$ will belong to a twisted cubic. This is also
immediate from (\ref{eq:leitkubik}). It is well known, that a twisted cubic can
be obtained as the intersection of \emph{three\/} quadrics in $B$. (We used two
quadratic cones projecting the cubic from two different points, and a
hyperbolic quadric.) Each of these three quadrics in $B$ gives rise to a
quadratic cone in $\PP_5(K)$ by joining it with the line $C^\perp$. The
quadratic polynomials in Theorem~\ref{thm:algebraisch} describe these three
quadratic cones in $\PP_5(K)$ so that $J$ actually is a cone with vertex
$C^\perp$ having a twisted cubic in $B$ as its base.

We wish to emphasize that our approach is motivated by the Thas-Walker
construction linking flocks of cones with spreads: The osculating planes
$\pi(s)$ yield a flock of the quadratic cone $C\cap Q$ if, and only if, the set
$\cO$ is a spread of $\PP_3(K)$. See \cite[pp.~334--338]{thas-95a} and
\cite[Theorem~6.2]{jha+j96a}.

It should be noted here that we did not use results from the thesis of R.~Koch
\cite[pp.~18--19]{koch-68}. He described, over the real numbers, the congruence
of osculating tangents in terms of a cubic and a quadratic form (equations
(2.27) and (2.28) loc.\ cit.). In fact, his equation (2.28) corresponds, up to
a change of coordinates, to our polynomial (\ref{eq:h1}). However, the two
equations of Koch are also satisfied by the Pl{\"u}cker coordinates of \emph{all\/}
lines through the point $Z$, whereas our system of equations yields less lines
through that point.

The next result says that, from an algebraic point of view, our system of
equations (\ref{eq:h1}) -- (\ref{eq:h3}) is the best possible:
\end{nrtxt}

\begin{thm}\label{thm:abgeschlossen}
Suppose that $K$ is an infinite field with $\Char K\neq 3$. Let $\cO$ and
$\cL[Z,\omega]$ be given as in Theorems~\emph{\ref{thm:S-eigenschaften}} and
\emph{\ref{thm:algebraisch}}, respectively. If $h(\bY)\in K[\bY]$ is a form
such that $\kappa(\cO)\subset\cV(h(\bY))$ then
$\kappa(\cL[Z,\omega])\subset\cV(h(\bY))$.
\end{thm}

\begin{proof}
Given such a form $h(\bY)\in K[X]$ we obtain from (\ref{eq:klein-schmiegtang})
the identity
\begin{equation}\label{eq:einsetz}
     h((1,3u_1,u_2,3u_1^2-u_2,u_1^3,3u_1^4-3u_1^2u_2+u_2^2)^\T)=0
     \mbox{ for all }(u_1,u_2)\in K^2.
\end{equation}
Due to $\Char K\neq 3$ there exists a field extension $\overline K/K$ of degree
$[\overline K:K]\leq 2$ containing a third root of unity $w\neq 1$. We infer
from a standard result on zeros of polynomials over an infinite domain (see,
for example, \cite[\S~28]{vand-66}) that (\ref{eq:einsetz}) holds also for all
$(u_1, u_2)\in\overline K^2$. We allow $u_1\in\overline K$ and replace $u_2$
with $(1-w)u_1^2+mu_1$ in (\ref{eq:klein-schmiegtang}), where $m\in \overline
K$ is fixed, but arbitrary. Thus (\ref{eq:klein-schmiegtang}) turns into
\begin{equation*}
    (1, 3u_1, ( 1-w ) u_1^{2}+m u_1, ( 2+w ) u_1^{2}-m u_1,
     u_1^{3}, -(2w+1) m u_1^{3} + {m}^{2} u_1^{2})^\T.
\end{equation*}
This is for $u_1\in \overline K$ a rational parametrization of all but one
points of a twisted cubic which is contained in the variety of $\PP_3(\overline
K)$ determined by $h(\bY)$ considered as an element of $\overline K[\bY]$. As
$\overline K$ is infinite, also the remaining point of the twisted cubic,
namely
\begin{equation*}
  \overline K(0,0,0,0,1,-(2w+1)m)^\T,
\end{equation*}
is a point of that variety. We claim that $2w+1\neq 0$: This is trivial when
$\Char K=2$. For $\Char K\neq 2$ the assertion holds, because our assumption
$\Char K\neq 3$ guarantees that $(-1/2)^3\neq 1$, whence $w\neq -1/2$.

Thus, for appropriate values of $m\in\overline K$, we see that all points of
the line $\kappa(\cL[Z,\omega])$ except $\kappa(g_\infty)$ belong to the
variety $\cV(h(\bY))$. Finally, since $K$ is infinite, we obtain
$\kappa(g_\infty)\in\cV(h(\bY))$.%
\end{proof}

\begin{cor}
Infinite Betten-Walker spreads are not algebraic sets of lines.
\end{cor}
%%%%%%%%%
We now turn to the remaining case of characteristic $3$. Here the situation is
completely different:

\begin{thm}\label{thm:par.netz}
Let $\Char K=3$. Then $\cO\cup\cL[Z,\omega]$ is a subset of a parabolic linear
congruence $\cN$ whose Klein image equals the quadratic cone $Q\cap D$, where
\begin{equation*}
      D:= \cV(X_{02},X_{03}+X_{12})
\end{equation*}
is a three-dimensional subspace of $\PP_5(K)$. The axis of the congruence $\cN$
is the line $n$ of nuclei. The congruence $\cN$ coincides with
$\cO\cup\cL[Z,\omega]$ if, and only if, each element of $K$ has a third root in
$K$.
\end{thm}

\begin{proof}
The polar subspace $D^{\perp}$ (with respect to the Klein quadric $Q$) is the
line joining
\begin{equation*}
    \kappa(n)=K(0,0,0,0,1,0)^\T \in Q
    \mbox{ and }
              K(0,0,1,1,0,0)^\T \not\in Q.
\end{equation*}
So $D^{\perp}\not\subset Q$ is a tangent of the Klein quadric and $D \cap Q$ is
a quadratic cone with vertex $\kappa(n)$. For all $(u_1,u_2) \in K^2$, the
Klein image of the only proper osculating tangent at $P(u_1,u_2)$ is the point
with coordinates
\begin{equation}\label{eq:klein-schmiegtang3}
  (1,0,u_2,-u_2,u_1^3,u_2^2)^\T.
\end{equation}
The Klein image of the pencil $\cL[Z,\omega]$ is the line spanned by
$\kappa(n)$ and $K\bw_\infty=\kappa(g_\infty)$; see (\ref{def:w.inf}). Now a
direct verification shows $(\cO\cup\cL[Z,\omega])\subset\cN$.

We read off from the penultimate coordinate in (\ref{eq:klein-schmiegtang3})
that $\cN=\cO\cup\cL[Z,\omega]$ holds precisely when every element of $K$ has a
third root in $K$.
\end{proof}
Of course, when $K$ is finite with characteristic $3$ then each of its elements
has a third root in $K$.

The line $D^{\perp}$ has yet another natural interpretation: Formula
(\ref{eq:erzeugendenbild}) can be rewritten in the form
\begin{equation}\label{oscu-tangent2-eq}
    \kappa(g(s_0,s_1))=K(s_0^3\bv_0+s_0^2s_1\bv_1+s_0s_1^2\bv_2+s_1^3\bv_3)
\end{equation}
with linearly independent $\bv_i \in K^{6 \times 1}$. By the above, we obtain a
twisted cubic as $(s_0,s_1) \ne (0,0)$ varies in $K^2$. Due to $\Char K=3$ all
osculating planes of this cubic belong to the pencil of planes (see
\cite[Theorem~21.1.2]{hirs-85}) with axis $K\bv_1 + K\bv_2 = D^{\perp}$.

%\bibliographystyle{plain} %acm amsplain amsplain
%\bibliography{d:/forschung/separata/ketten}

\noindent Hans Havlicek, Institut f{\"u}r Diskrete Mathematik und Geometrie,
Technische Universit{\"a}t Wien,
Wiedner Hauptstra{\ss}e 8--10/104, A-1040 Wien, Austria.\\
Email: havlicek@geometrie.tuwien.ac.at

\par~\par

\noindent Rolf Riesinger, Patrizigasse 7/14, A-1210 Wien, Austria.\\
Email: rolf.riesinger@chello.at

\end{document}